\documentclass[12pt]{article}
\usepackage{latexsym,amsmath,amssymb}
\usepackage[usenames]{color}

\newfont{\bb}{msbm10 at 12pt}

\def\r{\hbox{\bb R}}

\def\h{\hbox{\bb H}}

\def\s{\hbox{\bb S}}

\def\mkr{\mathbb{M}^2(\kappa)\times\mathbb{R}}
\def\hr{\mathbb{H}^2\times\mathbb{R}}
\def\m{\mathbb{M}^2}
\def\mk{\mathbb{M}^2(\kappa)}

\def\hm{\mathcal{M}(\kappa , \tau)}
\def\hmf{\hbox{\bb E}(\kappa , \tau)}
\def\amb{\mathcal{M}}
\newcommand{\campo}{\mathfrak{X}}
\newcommand{\camb}{\overline{\nabla}}
\newcommand{\norm}[1]{\left\Vert #1 \right\Vert}

\newcommand{\set}[1]{\left\{#1\right\}}
\newcommand{\meta}[2]{\langle #1,#2 \rangle }
\newcommand{\eps}{\varepsilon}
\newcommand{\ext}{\wedge}
\newcommand{\To}{\longrightarrow }

\usepackage[latin1]{inputenc}
\usepackage {amsmath}
\usepackage {amsthm}
\usepackage{times}
\usepackage{amscd}
\usepackage{epsf}
\begin{document}

\theoremstyle{plain}\newtheorem{lemma}{Lemma}[section]
\theoremstyle{plain}\newtheorem{definition}{Definition}[section]
\theoremstyle{plain}\newtheorem{theorem}{Theorem}[section]
\theoremstyle{plain}\newtheorem{proposition}{Proposition}[section]
\theoremstyle{plain}\newtheorem{remark}{Remark}[section]
\theoremstyle{plain}\newtheorem{corollary}{Corollary}[section]

\begin{center}

{\Large \bf Finite index operators on surfaces} \vspace{0.4cm}

\vspace{0.5cm}

{\large Jos$\acute{\text{e}}$ M. Espinar$\,^\dag$\footnote{The author is partially
supported by Spanish MEC-FEDER Grant MTM2010-19821, and Regional J. Andalucia Grants
P06-FQM-01642 and FQM325}}\\ 
\vspace{0.3cm} 
\end{center}

\vspace{.5cm}

\noindent $\mbox{}^\dag$ Departamento de Geometría y Topología, Universidad de
Granada, 18071 Granada, Spain; e-mail: jespinar@ugr.es\vspace{0.2cm}

\vspace{.3cm}

\begin{abstract}
We consider differential operators $L$ acting on functions on a Riemannian surface, $\Sigma$, of the form
$$L = \Delta + V -a K ,$$where $\Delta$ is the Laplacian of $\Sigma$, $K$ is the Gaussian curvature, $a$ is a positive constant and $V \in C^{\infty}(\Sigma)$. Such operators $L$ arise as the stability operator of $\Sigma$ immersed in a Riemannian three-manifold with constant mean curvature (for particular choices of $V$ and $a$).

We assume $L$ is nonpositive acting on functions compactly supported on $\Sigma$. If the potential, $V:= c + P $ with $c$ a nonnegative constant, verifies either an integrability condition, i.e. $P \in L^1(\Sigma)$ and $P$ is non positive, or a decay condition with respect to a point $p_0 \in \Sigma$, i.e. $|P(q)|\leq M/d(p_0,q)$ (where $d$ is the distance function in $\Sigma$), we control the topology and conformal type of $\Sigma$. Moreover, we establish a {\it Distance Lemma}.

We apply such results to complete oriented stable $H-$surfaces immersed in a Killing submersion. In particular, for stable $H-$surfaces in a simply-connected homogeneous space with $4-$dimensional isometry group, we obtain:
\begin{itemize}
\item There are no complete stable $H-$surfaces $\Sigma \subset \hr$, $H>1/2$, so that either
$K_e^+:={\rm max}\set{0,K_e} \in L^1 (\Sigma)$ or there exist a point $p_0 \in \Sigma$ and a constant $M$ so that $|K_e(q)| \leq M/d(p_0,q)$, here $K_e$ denotes the extrinsic curvature of $\Sigma$.
\item Let $\Sigma \subset \hmf $, $\tau \neq 0$, be an oriented complete stable $H-$surface so that either $\nu^2 \in L^1 (\Sigma)$ and $4H^2 +\kappa \geq 0$, or there exist a point $p_0 \in \Sigma$ and a constant $M$ so that $|\nu (q)|^2 \leq M/d(p_0,q)$ and $4H^2 +\kappa > 0$. Then:
\begin{itemize}
\item In $\s ^3_{Berger}$, there are no such a stable $H-$surface.
\item In ${\rm Nil}_3$, $H=0$ and $\Sigma$ is either a vertical plane (i.e. a vertical cylinder over a
straight line in $\r^2$) or an entire vertical graph.
\item In $\widetilde{{\rm PSL}(2,\r)}$, $H=\sqrt{-\kappa}/2$ and $\Sigma$ is either a vertical
horocylinder (i.e. a vertical cylinder over a horocycle in $\h ^2 (\kappa)$) or an entire graph.
\end{itemize}
\end{itemize}
\end{abstract}

\vspace{.2cm}


\vspace{.2cm}


\section{Introduction}
Stable oriented domains $\Sigma$ on a constant mean curvature surface in a
Riemannian three-manifold $\amb ^3$, are characterized by the \emph{stability inequality}
for normal variations $\psi N$ (see \cite{BdCE})

$$ \int _{\Sigma} \psi ^2 |A|^2 + \int _{\Sigma} \psi ^2 {\rm Ric}_{\mathcal{M}^3} (N,N)
\leq \int_{\Sigma} |\nabla \psi|^2  ,$$for all compactly supported functions $\psi
\in H^{1,2}_c (\Sigma)$. Here $|A|^2$ denotes the square of the length of the second
fundamental form of $\Sigma$, ${\rm Ric}_{\mathcal{M}^3} (N,N)$ is the Ricci
curvature of $\mathcal{M}^3$ in the direction of the normal $N$ to $\Sigma$ and
$\nabla $ is the gradient w.r.t. the induced metric.

One writes the stability inequality in the form
$$ \left.\frac{d^2}{dt^2}\right\vert _{t=0}\left({\rm Area}(\Sigma (t))-2H \, {\rm Vol}(\Sigma (t))\right)=
- \int _{\Sigma} \psi L \psi \geq 0 ,$$where $L$ is the linearized operator of the
mean curvature
$$ L = \Delta + |A|^2 + {\rm Ric}_{\amb} .$$

In terms of $L$, stability means that $-L$ is nonnegative, i.e., all its eigenvalues
are non-negative. $\Sigma$ is said to have finite index if $-L$ has only finitely
many negative eigenvalues.


The study of stable surfaces by considering Schr\"{o}dinger-type differential operators on a
surface $\Sigma$ with a metric $g$ of the form
$$L:= \Delta + V - a K,$$where $\Delta $ and $K$ are the Laplacian and Gaussian curvature
associated to $g$ respectively, $a $ is a positive constant and $V$ is a nonnegative function,  has been received an important number of contributions (see \cite{dCP,FCS,FC,Gu1,Gu2,K,P}), and even now it is a topic of interest. T. Colding and W. Minicozzi \cite{CM} introduced a new technique to study this type of operator
based on the first variation formula for length and the Gauss-Bonnet formula. Using
this technique they obtained an inequality which, when $a>1/2$, gives quadratic area
growth of the geodesic disks on the surface and the integrability of the potential
$V$ at the same time (note that the stability operator can be realized with the
right choice of $V$ and $a$). P. Castillon \cite{Ca} used the ideas of
Colding-Minicozzi to improve their result to $a>1/4$. This technique allows them to
control the topology and the conformal type of the surface. In the above works, the
{\it potential}, $V$, is assumed to be nonnegative. Moreover, an important result
for this kind of operators is the {\it Distance Lemma}, technique developed by
Fischer-Colbrie \cite{FC}, which bounds the intrinsic distance of any point in the
surface to the boundary. This result have been done for $V\geq c$, $c$ a positive
constant, and $a>1/4$ (see \cite{MPR} for a survey). In \cite{ER}, the authors
extended the {\it Distance Lemma} for nonnegative Schr\"{o}dinger-type differential operators satisfying $V \geq c$, $c $ a positive constant, $0< a \leq 1/4$, and assuming some control on
the area growth of the geodesic disks by different methods. Also, they were able to
control the topology of the surface. As we have mentioned, all these results depend
on conditions on the potential $V$ and the constant $a$. Recently,
Manzano-P\'{e}rez-Rodr\'{i}guez \cite{MaPRo} have imposed no condition on the
potential, $\mathcal Q := V - aK \in C^{\infty}(\Sigma)$, but $\Sigma$ is a complete
parabolic\footnote{A Riemannian manifold $\Sigma$ is parabolic if every positive
subharmonic function on $\Sigma$ must be constant.} surface with no boundary, and
they obtained that, if the there exists a nonidentically zero {\it bounded} solution
of $L f =0$, $-L$ nonnegative on $\Sigma$, then $f$ vanishes nowhere and the linear
subspace of such functions is one dimensional.

In this paper, we drop the condition $V \geq 0$ for either the integrability of the potential or some decay at infinity. We will make them explicit. Those conditions allow us to obtain parabolicity of the surface or even a {\it Distance Lemma}.

The above achievements for Schr\"{o}dinger-type operators have been used for proving results for stable $H-$surface in three-manifolds (see \cite{ER,Ma,MPR,MaPRo,R,R2} and references therein). The study of stable $H-$surface in a simply-connected homogeneous three-manifold with a four dimensional isometry group is a topic of increasing interest. These homogeneous
spaces are denoted by $\hmf $, where $\kappa$ and $\tau$ are constant and $\kappa
-4\tau ^2 \neq 0$. They can be classified as $\mkr$ if $\tau = 0$, with $\mk = \s
^2(\kappa)$ if $\kappa >0$ ($\s ^2(\kappa)$ the sphere of curvature $\kappa$), and
$\mk  =  \h ^2(\kappa)$ if $\kappa < 0$ ($\h^2(\kappa)$ the hyperbolic plane of
curvature $\kappa$).  If $\tau$ is not equal to zero, $\hmf $ is a Berger sphere if
$\kappa > 0$, a Heisenberg space if $\kappa = 0$ (of bundle curvature $\tau$), and
the universal cover of ${\rm PSL}(2,\r)$ if $\kappa < 0$.

We apply our results to stable $H-$surfaces immersed in a Riemannian three-manifold
which fiber over a Riemmanian surface and whose fibers are the trajectories of a unit Killing vector field (see \cite{ES,LR,RST} and references therein). In particular, they include the simply-connected homogeneous spaces $\hmf$.

The paper is organized as follows. Section \ref{Sec:Preli} is devoted to establish
the notation and basic concepts. In Section \ref{Sec:Integrable} we study
nonnegative differential operators with {\bf integrable potential}, this means that
$L_{a,c}:= \Delta + V - a K$ satisfies that $V:= c + P $, where $c $ is a
nonnegative constant and $P $ is a nonpositive and integrable function on $\Sigma$,
i.e., $P \in L^1(\Sigma)$ (see Definition \ref{Def:integrablepotential}).

When $a>1/4$, for this kind of operators we get:

\begin{quote}
{\bf Theorem \ref{Theo:QAGintegrable}.} {\it Let $\Sigma $ be a Riemannian surface possibly with boundary. Suppose that $L_{a,c} = \Delta + V - a K$ is nonpositive acting on $f\in C^\infty _0(\Sigma)$ and has integrable potential with $c \geq 0 $ and $a>1/4$. Then, $\Sigma$ has Quadratic Area Bound (the area bound depending only on $a$, $c$ and $\norm{P}_1$).

Moreover, if we assume $\Sigma$ is complete (without boundary), $\Sigma$ is
conformally equivalent to a compact Riemann surface with a finite number of points
removed.}
\end{quote}

And, in the case $c>0$, we can go further. First, we shall introduce a concept to understand correctly the next theorem. Let $\Sigma $ be a Riemannian surface with boundary $\partial \Sigma$, we say that the area of the geodesic disks goes to infinity as its radius goes to infinity if for any point $p \in \Sigma$ and any $s >0$ so that $\overline{D(p , s)} \cap \partial \Sigma = \emptyset$, where $D(p,s)$ is the geodesic disk in $\Sigma$ centered at $p$ and radius $s$, the function 
$$ a(p,s) := {\rm Area}(D(p, s)) ,$$goes to infinity if $s$ goes to infinity. Now, we are ready for establishing:

\begin{quote}
{\bf Theorem \ref{Theo:distancegeq}.} {\it Let $\Sigma $ be a Riemannian surface
possibly with boundary. Suppose that $L_{a,c} = \Delta + V - a K$ is nonpositive, has integrable potential with $c >0 $ and $a>1/4$. Then, if the area of the
geodesic disks goes to infinity as its radius goes to infinity, there exists a
positive constant $C$ such that
$$ {\rm dist}_{\Sigma}(p , \partial \Sigma) \leq C , \, \, \forall p \in \Sigma .$$

In particular, if $\Sigma$ is complete with $\partial \Sigma = \emptyset$, then it must be either  compact or parabolic with finite area.  Moreover, when $\Sigma$ is compact, it holds
$$ c \, A(\Sigma) - \norm{P}_1 \leq 2a \pi \, \chi (\Sigma) ,$$where $A(\Sigma)$ and $\chi (\Sigma)$ denote the area and Euler characteristic of $\Sigma$ respectively.}
\end{quote}

When $0< a \leq 1/4$, we obtain:

\begin{quote}
{\bf Theorem \ref{Theo:distance}.} {\it Let $\Sigma $ be a Riemannian surface with
$k-AAB$ (see Definition \ref{Def:kAAG}) and possibly with boundary. Suppose that $L_{a,c} = \Delta + V - a K$ is nonpositive, has integrable potential with $c >0 $ and $0< a \leq 1/4$. Then, there exists a positive constant $C$ such that
$$ {\rm dist}_{\Sigma}(p , \partial \Sigma) \leq C , \, \, \forall p \in \Sigma .$$

In particular, if $\Sigma$ is complete with $\partial \Sigma = \emptyset$ then it
must be compact.  Moreover, it holds
$$ c \, A(\Sigma) - \norm{P}_1 \leq 2a \pi \, \chi (\Sigma) ,$$where $A(\Sigma)$ and $\chi (\Sigma)$ denote the area and Euler characteristic of $\Sigma$ respectively.}
\end{quote}

In Section \ref{Sec:Decay}, we impose that the potential has linear decay with respect some point, specifically, $L_{a,c} := \Delta + V - a K$ has {\bf linear decay} if $V:= c + P$, where $c $ is a nonnegative constant and $P $ satisfies
\begin{equation*}
|P(q)|\, \leq M/d(p_0,q) , 
\end{equation*}for some point $p_0\in \Sigma$, where $M$ is a nonnegative constant (see
Definition \ref{Def:QDecay}). We prove

\begin{quote}
{\bf Theorem \ref{Theo:QDC}.} {\it Let $\Sigma $ be a complete Riemannian surface without boundary. Suppose that $L_{a,c} = \Delta + V - a K$ is nonpositive acting on $f\in C^\infty _0(\Sigma)$, has linear decay with $c > 0 $ and $a>1/4$. Then, $\Sigma$ is compact.}
\end{quote}

Next, in Section \ref{Sec:stable}, we apply these abstract results to stable $H-$surfaces immersed in a Riemannian three-manifold. In particular, for stable $H-$surfaces in a
Killing submersion

\begin{quote}
{\bf Theorem \ref{Theo:stablehm}.} {\it Let $\Sigma $ be a complete oriented $H-$surface with finite index immersed in $\hm$, $\hm$ an orientable Killing submersion of bounded geometry so that $4H^2 +c(\Sigma)\geq 0$, where
\begin{equation*}
c(\Sigma) := {\rm inf}\set{\kappa (\pi (p)) : \, p \in \Sigma} .
\end{equation*}

Set
\begin{eqnarray*}
P^- &:=& {\rm min}\set{0 , -(K_e +\tau^2) } , \\
P^+ &:=& {\rm max}\set{0 , -(K_e +\tau^2)  }.
\end{eqnarray*}

Assume $P^-\in L^1(\Sigma)$. Then, one of the following statements hold:
\begin{itemize}
\item $\Sigma$ is a minimal graph with $\pi(\Sigma)=\m$ and $c(\Sigma)>0$,
\item $4H^2 + c(\Sigma)=0$ and $\Sigma$ is either a vertical multigraph or a vertical cylinder of geodesic curvature $2H$ in $\m$.
\end{itemize}

}
\end{quote}

And

\begin{quote}
{\bf Theorem \ref{Theo:stablehmQCD}.} {\it Let $\Sigma $ be a complete oriented stable $H-$surface in $\hm$, $\hm$ an orientable Killing submersion of bounded geometry so that $4H^2 +c(\Sigma)\geq 0$, where
\begin{equation*}
c(\Sigma) := {\rm inf}\set{\kappa (\pi (p)) : \, p \in \Sigma} .
\end{equation*}

Set $P := K_e +\tau ^2$ and assume there exist a point $p_0 \in \Sigma$ and a constant $M>0$ so that
\begin{equation*}
|P(q)| \leq M/d(p_0,q).
\end{equation*}

Then, $\Sigma$ is a minimal graph with $\pi(\Sigma)=\m$ and $c(\Sigma) >0$.}
\end{quote}

But, if we restrict the above results to the three-dimensional simply-connected homogeneous spaces, we obtain:

\begin{quote}
{\bf Corollary \ref{Cor:hr}.} {\it Let $\Sigma \subset \hr$ be an oriented complete stable $H-$surface satisfying one of the following conditions:
\begin{itemize}
\item $H\geq 1/2$ and ${\rm max}\set{0,K_e} \in L^1 (\Sigma)$,
\item $H>1/2$ and there exists a point $p_0 \in \Sigma$ and a constant $M>0$ so that
\begin{equation*}
|K_e (q)| \leq M/d(p_0,q).
\end{equation*}
\end{itemize}

Then, $H=1/2$ and $\Sigma$ is either a vertical horocylinder (i.e. a vertical
cylinder over a horocycle in $\h^2$) or an entire vertical graph.}
\end{quote}

\begin{quote}
{\bf Corollary \ref{Cor:homog}.} {\it Let $\Sigma \subset \hmf $, $\tau \neq 0$, be an oriented complete stable
$H-$surface satisfying one of the following conditions:
\begin{itemize}
\item $4H^2 +\kappa \geq 0$ and $\nu^2 \in L^1 (\Sigma)$,
\item $4H^2 +\kappa > 0$ and there exist a point $p_0 \in \Sigma$ and a constant $M>0$ so that
\begin{equation*}
|\nu (p)|^2 \leq M/d(p_0,q).
\end{equation*}
\end{itemize}

Then:
\begin{itemize}
\item In $\s ^3_{Berger}$, there are no such a stable $H-$surface.
\item In ${\rm Nil}_3$, $H=0$ and $\Sigma$ is either a vertical plane (i.e. a vertical cylinder over a
straight line in $\r^2$) or an entire vertical graph.
\item In $\widetilde{{\rm PSL}(2,\r)}$, $H=\sqrt{-\kappa}/2$ and $\Sigma$ is either a vertical
horocylinder (i.e. a vertical cylinder over a horocycle in $\h ^2 (\kappa)$) or an entire graph.
\end{itemize}}
\end{quote}

Finally, in the Appendix, we have compiled a sort of results we will make use along
this paper for the sake of completeness.

\section{Notation and preliminary results}\label{Sec:Preli}

Throughout this work, we denote by $\Sigma$ a connected Riemannian surface, with
Riemannian metric $g$, and possibly with boundary $\partial \Sigma$. Let $p \in
\Sigma$ be a point of the surface and $D(p,s)$, for $s>0$, denote the geodesic
disk centered at $p$ of radius $s$. We assume that $\overline{D(p ,s)} \cap
\partial \Sigma = \emptyset$. Moreover, let $r$ be the radial distance of a point
$q$ in $D(p, s)$ to $p$. We write $D(s)=D(p ,s)$ if we can omit the dependence on $p$. We also denote
\begin{eqnarray*}
l(s) &:=& {\rm Length}(\partial D(s)) \\
a(s) &:=& {\rm Area}(D(s))\\
K(s) &:=& \int _{D(s)} K \\
\chi (s)&:=& \text{Euler characteristic of } D(s) ,
\end{eqnarray*}where $K$ is the Gaussian curvature associated to the metric $g$. In the case we can not drop the dependence on $p$, we write $l(p,s)$, $a(p,s)$, $K(p,s)$ and $\chi (p,s)$ respectively.

Let
$$L:= \Delta + V - a K$$be a differential operator on $\Sigma$ acting on piecewise
smooth functions with compact support, i.e. $f \in C^\infty _0 (\Sigma)$, where $a
>0$ is a constant, $V \in C^{\infty}(\Sigma)$ and $\Delta$ is the Laplacian operator
associated to the metric $g$.

The index form of these kind of operators is
\begin{equation}\label{varcharacsurface}
I(f)=\int _{\Sigma } \left\{ \norm{\nabla f}^2 - V f^2+ a K f^2 \right\} ,
\end{equation}where $\nabla $ and $\| \cdot \|$ are the gradient and norm associated
to the metric $g$. One has
$$\int_{\Sigma} f L f = -I (f) .$$

Thus, the nonpositivity of $L$ implies that the quadratic form, $I(f)$, associated to $L$ is nonnegative on compactly supported functions, i.e., $I(f) \geq 0$. In this case, we will say that $-L$ is {\bf stable}.

$-L$ is stable if all the eigenvalues of $-L$ are nonnegative. $-L$ has {\bf finite
index} if has only finitely many negative eigenvalues. In this case, it is well
known that there exists a compact set $K\subset \Sigma$ so that $-L$ is nonnegative
acting on $f \in C^{\infty}_0(\Sigma \setminus K)$ (see \cite{FC}).

We recall now some topological concepts. For a compact surface $\Sigma$, its Euler
characteristic is given by $\chi (\Sigma) := 2(1-g_{\Sigma})-n_{\Sigma}$, where
$g_{\Sigma}$ and $n_{\Sigma}$ denote its genus and the number of connected
components of its boundary respectively.

A noncompact surface $\Sigma$ is said to be of finite topology if there exists a
compact surface $\tilde \Sigma$ without boundary and a finite number of pairwise
disjoint closed disks $D_i \subset \tilde \Sigma$, $i=1, \ldots , n$, so that
$\Sigma$ is homeomorphic to $\tilde \Sigma \setminus \bigcup _{i=1}^n D_i$. In this
case, the Euler characteristic of $\Sigma$ is $\chi (\Sigma)= 2(1-g_{\tilde
\Sigma})-n$.

Moreover, we will need the following topological result (see \cite[Lemma 1.4]{Ca})

\begin{lemma}\label{Lem:Castillon}
Let $\Sigma$ be a complete Riemannian surface.
\begin{itemize}
\item If $\Sigma$ is of finite topology, then there exists $s_0$ such that for all
$s \geq s_0$ we have $\chi (s) \leq \chi (\Sigma)$.
\item If $\Sigma$ is not of finite topology then $\lim _{s \to +\infty}\chi
(s)=-\infty$.
\end{itemize}
\end{lemma}

\section{Nonpositive operators with integrable potential}\label{Sec:Integrable}

In this Section we study differential operators with integrable potential on a
Riemannian surface. First, we make explicit the kind of differential operators we
are interested on:

\begin{definition}\label{Def:integrablepotential}
Let $\Sigma$ be a Riemannian surface. We say that $L_{a,c} = \Delta + V - a K$ has
{\bf integrable potential} if $L_{a,c}$ is a differential operator on $\Sigma$
acting on piecewise smooth functions with compact support, i.e. $f \in C^{\infty}_0
(\Sigma)$, where $a > 0$ is constant, $\Delta$ and $K$ are the Laplacian and Gauss
curvature associated to the metric $g$ respectively. Moreover, we will assume that
$V:= c + P $, where $c $ is a nonnegative constant and $P $ is a nonpositive and
integrable function on $\Sigma$, i.e., $P \in L^1(\Sigma)$ or, equivalently, $\norm{P}_1 < + \infty$, where $\norm{\cdot } _1$ denotes the $L^1-$norm.
\end{definition}

We will use the following condition on the area growth of $\Sigma$.

\begin{definition}\label{Def:AAG}
Let $\Sigma$ be a Riemannian surface possibly with boundary. We say that $\Sigma$ has Asymptotic Area Bound of degree $k$ ($k-$AAB) if there exist positive constants $k ,C \in \r ^+ $ such that for any point $p \in \Sigma$ and any $s >0$ so that $\overline{D(p , s)} \cap \partial \Sigma = \emptyset$, the function 
$$ \frac{a (p, s) }{s^k}:= \frac{{\rm Area}(D(p, s))}{s^k} \to C \text{ if } s \to + \infty .$$

If $\Sigma$ is complete without boundary and for some point $p \in \Sigma $ verifies
$$ \lim _{s\To \infty}\frac{{\rm Area} ( D(p,s))}{s^k} = C  ,$$we say that $\Sigma$ has Asymptotic Area Growth of degree $k$ ($k-$AAG). Note that, by the Triangle Inequality, this condition does not depend on the point $p$.
\end{definition}

When $-L$ is nonnegative, $a>1/4$ and $V \geq 0$, it is known that $\Sigma$ has quadratic area growth and the integrability of the potential (see \cite{Ca}, \cite{MPR} for $a>1/4$ or
\cite{CM}, \cite{R} for $a>1/2$). When $a \leq 1/4$, $V \geq 0$ and we assume that
$\Sigma$ has asymptotic area growth of degree $k$, we can obtain similar results
(see \cite{ER}). 

One of the most interesting consequence under the above conditions is that one can
bound the distance of any point to the boundary, this is known as the \emph{Distance
Lemma} (see \cite{MPR} or \cite{ER}). the {\it Distance Lemma} allows us to conclude
that, if $\Sigma$ is complete, it is compact.

We distinguish two cases depending on the value of $a$. We start when $a > 1/4$, and
first we will prove that the surface has Quadratic Area Bound.

\begin{definition}\label{Def:QAG}
Let $\Sigma$ be a Riemannian surface possibly with boundary. We say that $\Sigma$ has Quadratic Area Bound if there exists a positive constant $C$ so that for any point $p \in \Sigma$ and any $s >0$ so that $\overline{D(p , s)} \cap \partial \Sigma = \emptyset$, the function 
$$ a(p,s/2) \leq C s^2 .$$

If $\Sigma$ is complete without boundary and for some point $p \in \Sigma $ verifies
$$ a(p,s) \leq C s^2, \, \text{ for all } s>0,$$we say that $\Sigma$ has Quadratic Area Growth. Note that, by the Triangle Inequality, this condition does not depend on the point $p$.
\end{definition}

Actually, this will give us more information about the topology and conformal type
of the surface.

\begin{theorem}\label{Theo:QAGintegrable}
Let $\Sigma $ be a Riemannian surface
possibly with boundary. Suppose that $L_{a,c} = \Delta + V - a K$ is nonpositive acting on $f\in C^\infty _0(\Sigma)$, has integrable potential with $c \geq 0 $ and $a>1/4$. Then, $\Sigma$ has Quadratic Area Bound (the area bound depending only on $a$, $c$ and $\norm{P}_1$).

Moreover, if we assume $\Sigma$ is complete (without boundary), $\Sigma$ is conformally equivalent to a compact Riemann surface with a finite number of points removed.
\end{theorem}
\begin{proof}
Since $a >1/4$, take $b \geq 1 $ in \eqref{MPRequation} so that
$$ - \alpha := b(b(1-4a)+2a) <0 .$$

Thus, applying Corollary \ref{Cor:ageq14}, we obtain
\begin{equation}\label{th1:eq1}
\frac{\alpha}{s^2}\int _0 ^s (1-r/s)^{2b-2}l(r) \leq 2a \pi - \int _{D(s)}
(1-r/s)^{2b} P ,
\end{equation}where $P$ is a nonpositive integrable function, and so
$$ - \int _{D(s)} (1-r/s)^{2b} P \leq C  \, \text{ for all } \, s >0, $$where $C$ is some positive constant depending only on $\norm{P}_1$. Hence, since
$$ \frac{\alpha}{s^2}\int _0 ^s (1-r/s)^{2b-2}l(r) \geq \frac{\alpha}{2^{2b-2}s^2} a(s/2) ,
$$we obtain, inserting the above inequalities in \eqref{th1:eq1},
\begin{equation*}
\frac{\alpha}{2^{2b-2}s^2}a(s/2) \leq 2a \pi + C ,
\end{equation*}then
$$ a(s/2) \leq \tilde C s ^2 ,$$where $\tilde C$ is some positive constant depending only on $a$, $c$ and $\norm{P}_1$. This holds for any $p \in \Sigma \setminus \partial \Sigma$ and hence, $\Sigma $ has Quadratic Area Bound.

Now, we assume $\Sigma$ is complete. Take $b \geq 1$ so that $-\alpha :=
b(b(1-4a)+2a)<0$. Then, applying Corollary \ref{Cor:ageq14}, we have
\begin{equation}\label{th1:eq2}
0 \leq \frac{\alpha}{s^2}\int _0 ^s (1-r/s)^{2b-2}l(r) \leq 2a \pi G(s)- \int
_{D(s)} (1-r/s)^{2b} P \leq 2 a \pi G(s) + C,
\end{equation}where $C $ is some nonnegative constant depending only on $\norm{P}_1$.

Let us prove first that $\Sigma$ has finite topology. Assume $\Sigma $ has infinite topology, then
$$ \liminf _{s\to + \infty} \chi (s) = -\infty ,$$and hence, there exists $s_0$ so that for
all $s \geq s_0$, we have $\chi (s) \leq - M $, where $M:=\frac{C+1}{2a\pi}$. Therefore
\begin{equation*}
\begin{split}
G(s) & = - \int _0 ^s (f(r)^2)' \chi (r) = - \int _0 ^{s_0} (f(r)^2)' \chi (r)- \int
_{s_0} ^s (f(r)^2)' \chi (r)\\
 & \leq - \int _0 ^{s_0} (f(r)^2)' + M \int _{s_0} ^s (f(r)^2)' =  -\left( f(s_0) ^2 - f(0)^2\right) +
 M \left( f(s)^2 - f(s_0)^2 \right) \\
 &= - (M+1) f(s_0)^2 +1 = -(M+1)\left( 1-s_0/s\right)^{2b} +1 ,
\end{split}
\end{equation*}and, letting $s\to + \infty$ in the above expression, we can see that
$$ \lim _{s\to +\infty} G(s) \leq - M ,$$and so
$$ \lim_{s\to + \infty} 2a \pi G(s) + C \leq -1 ,$$which contradicts \eqref{th1:eq2}.
Therefore, $\Sigma$ has finite topology.

Thus, since $\Sigma$ has finite topology and QAG, $\Sigma$ is conformally
equivalent to a compact Riemann surface with a finite number of points removed.
\end{proof}

For the next result recall that, for a Riemannian surface $\Sigma $ with boundary $\partial \Sigma$, we say that the area of the geodesic disks goes to infinity as its radius goes to infinity if for any point $p \in \Sigma$ and any $s >0$ so that $\overline{D(p , s)} \cap \partial \Sigma = \emptyset$, where $D(p,s)$ is the geodesic disk in $\Sigma$ centered at $p$ and radius $s$, the function 
$$ a(p,s) := {\rm Area}(D(p, s)) ,$$goes to infinity if $s$ goes to infinity. Now, we can prove

\begin{theorem}\label{Theo:distancegeq}
Let $\Sigma $ be a Riemannian surface possibly with boundary. Suppose that $L_{a,c} = \Delta + V - a K$ is nonpositive, has integrable potential, $V:=P+c$, with $c >0 $ and $a>1/4$. Then, if the area of the geodesic disks goes to infinity as its radius goes to infinity, there exists a positive constant $C$ (depending only on $a$, $c$ and $\norm{P}_1$) such that
$$ {\rm dist}_{\Sigma}(p , \partial \Sigma) \leq C , \, \, \forall p \in \Sigma .$$

In particular, if $\Sigma$ is complete with $\partial \Sigma = \emptyset$, then it must be either compact or parabolic with finite area. Moreover, when $\Sigma$ is compact, it holds
$$ c \, A(\Sigma) - \norm{P}_1 \leq 2a \pi \, \chi (\Sigma) ,$$where $A(\Sigma)$ and $\chi (\Sigma)$ denote the area and Euler characteristic of $\Sigma$ respectively.
\end{theorem}
\begin{proof}
Let us suppose that the distance to the boundary were not bounded. Then there exists
a sequence of points $\set{p_i} \in \Sigma$ such that ${\rm dist}^{\Sigma}(p_i ,
\partial \Sigma) \To + \infty$. So, for each $p_i$ we can choose a real number $s_i$
such that $s_i \To + \infty$ and $\overline{D(p_i ,s_i)}\cap \partial \Sigma =
\emptyset $.

We argue as in Theorem \ref{Theo:QAGintegrable}. Take $b\geq 1$ so that $- \alpha := b(b(1-4a)+2a)<0$. Then, applying \eqref{MPRequation} to each disk $D(p_i ,s_i )$, we have
\begin{equation*}
c \, a(p_i , s_i /2) \leq C ,
\end{equation*}where $C$ is constant independing only on $a$, $c$ and $\norm{P}_1$.

Now, bearing in mind that the left hand side of the above inequality goes to
infinity and the right hand side remains bounded, we obtain a contradiction.

Also, if $\Sigma$ is complete and has not finite area, the above estimate and the
Hopf-Rinow Theorem imply that $\Sigma $ must be compact.

When $\Sigma $ is compact, the last formula follows taking the constant function $f \equiv 1$ on $\Sigma$.
\end{proof}

Now, we focus on the case $0< a \leq 1/4 $.

\begin{theorem}\label{Theo:distance}
Let $\Sigma $ be a Riemannian surface with $k-AAB$ and possibly with boundary.
Suppose that $L_{a,c} = \Delta + V - a K$ is nonpositive, has integrable potential with $c >0 $ and $0< a \leq 1/4$. Then, there exists a positive constant $C$ such that
$$ {\rm dist}_{\Sigma}(p , \partial \Sigma) \leq C , \, \, \forall p \in \Sigma .$$

In particular, if $\Sigma$ is complete with $\partial \Sigma = \emptyset$ then it
must be compact.  Moreover, it holds
$$ c \, A(\Sigma) - \norm{P}_1 \leq 2a \pi \, \chi (\Sigma) ,$$where $A(\Sigma)$ and $\chi (\Sigma)$ denote the area and Euler characteristic of $\Sigma$ respectively.
\end{theorem}
\begin{proof}
Let us suppose that the distance to the boundary were not bounded. Then there exists
a sequence of points $\set{p_i} \in \Sigma$ such that ${\rm dist}^{\Sigma}(p_i ,
\partial \Sigma) \To + \infty$. So, for each $p_i$ we can choose a real number $s_i$
such that $s_i \To + \infty$ and $\overline{D(p_i ,s_i)}\cap \partial \Sigma =
\emptyset $.

Take $p \in \Sigma$ and $s>0$ so that $\overline{D(p,s)} \cap \partial \Sigma = \emptyset$. Let $f(r)$ be the radial function given by \eqref{function}. Set
$$ \omega :=  \int _{D(s e^{-s})}  V  + \int _{D(s)\setminus D(s e^{-s})} \left(
\frac{\ln(s/r)}{s}\right) ^{2b} V ,$$then
\begin{equation*}
\begin{split}
\omega &= c \int _{D(s)}f(s)^2 + \int _{D(s)} P f(s)^2 \geq c \int _{D(s)} f(s)^2 +
\int _{D(s)}P \\
 &\geq  c\int _{D(s)}f(s)^2 - \tilde{C},
\end{split}
\end{equation*}where $\tilde C$ is a nonnegative constant depending only on $\norm{P}_1$.

Now, let $\beta \in \r $ be a real number greater than one, then
\begin{equation*}
\int _{D(s)}f(s)^2 = \int _{D(s e^{-s})}  1 + \int _{D(s)\setminus D(s e^{-s})} \left(
\frac{\ln(s/r)}{s}\right) ^{2b} \geq  \frac{\beta ^{2b}}{s^{2b}}a(s e^{-\beta}).
\end{equation*}

Thus, joining the above inequalities we get
\begin{equation*}
\int _{D(s e^{-s})}V  + \int _{D(s)\setminus D(s e^{-s})} \left(
\frac{\ln(s/r)}{s}\right) ^{2b}V \geq c \frac{\beta ^{2b}}{s^{2b}}a(s e^{-\beta}) -
\tilde{C}.
\end{equation*}

Now, choose $b > 1$ such that $2(b+1) \geq k > 2b > 2 $. Thus, by Corollary
\ref{Cor:aleq14} and the above inequality

\begin{equation}\label{eq1}
c \frac{\beta ^{2b}}{s^{2b}} a(s e^{-\beta}) \leq C  +\rho^{+}_{a,b}(\delta _0 ,s),
\end{equation}where $C$ is a positive constant depending on $a$ and $\norm{P}_1$.

Now, since $\Sigma $ has $k-$AAB and $k>2b$ then for $s$ large enough we have that
the left hand side goes to infinity
\begin{equation}\label{eq2}
c \frac{\beta ^{2b}}{s^{2b}} a(s e^{-\beta}) \sim s^{k-2b} \To + \infty ,
\end{equation}but the right hand side remains bounded (see the asymptotic properties
of $\rho ^+ $ in Corollary \ref{Cor:aleq14}).

Thus, applying \eqref{eq1} to each disk $D(p_i ,s_i )$, and bearing in mind that,
from \eqref{eq2}, the left hand side of \eqref{eq1} goes to infinity, and the right
hand side remains bounded, we obtain a contradiction.

We still have to consider the case $k\leq 2$. Here, we use Corollary
\ref{Cor:ageq14}. Thus, for $b=1$ and the $k-$AAB, $k\leq 2$, of $\Sigma$, the right
hand side of \eqref{MPRequation} remains bounded by some positive constant $C$ (depending on the $k-$AAB) as $s$ goes to infinity. But,
$$ \int _{D(s)} \left( 1- r/s\right)^{2b} V \geq c \frac{a(s /2)}{4} - \tilde{C},
$$for a nonnegative constant $\tilde C$ depending only on $a$ and $\norm{P}_1$. Thus, we obtain
\begin{equation}\label{eqcontra}
c \, a(s/2) \leq C  \, \text{ for all } \, s >0,
\end{equation}where $C$ is a constant depending on $a$, $\norm{P_1} $ and the $k-$AAB.

Applying \eqref{eqcontra} to each disk $D(p_i ,s_i )$, and bearing in mind that the
left hand side of \eqref{eqcontra} goes to infinity and the right hand side remains
bounded, we obtain a contradiction.

Now, if $\Sigma$ is complete, then the above estimate and the Hopf-Rinow Theorem
imply that $\Sigma $ must be compact. The last formula follows by taking the constant function $f \equiv 1$ as above. 
\end{proof}

We have worked with nonnegative differential operators $-L_{a,c}$. This means that
all the eigenvalues of $-L_{a,c}$ are nonnegative. $-L_{a,c}$ has finite index if
has only finitely many negative eigenvalues or, equivalently, if there exists a compact set $K\subset \Sigma$ so that $-L_{a,c}$ is nonnegative acting on $f \in C^{\infty}_0(\Sigma \setminus K)$ (see \cite{FC}). As a consequence of these proofs we have the following

\begin{corollary}\label{Cor:QAG}
Under the hypothesis of Theorem \ref{Theo:QAGintegrable}, if $-L_{a,c}$, $c=0$, has
finite index and $\Sigma$ is complete with $\partial \Sigma = \emptyset$, then $\Sigma$ is conformally equivalent to a compact Riemann surface with a finite number of points removed.
\end{corollary}

And

\begin{corollary}\label{Cor:distance}
Under the hypothesis of Theorem \ref{Theo:distancegeq} or Theorem \ref{Theo:distance}, if $-L_{a,c}$ has finite index and $\Sigma$ is complete with $\partial \Sigma = \emptyset$, then it must be either compact or parabolic with finite area.
\end{corollary}

\section{Non positive operators with linear decay}\label{Sec:Decay}

We follow the notation of the previous section. First, let us make explicit the
operators we will work with.

\begin{definition}\label{Def:QDecay}
Let $\Sigma$ be a Riemannian surface. We say that $L_{a,c} = \Delta + V - a K$ has
{\bf Linear Decay} if $L_{a,c}$ is a differential operator on $\Sigma$ acting on
piecewise smooth functions with compact support, i.e. $f \in C^{\infty}_0 (\Sigma)$,
where $a > 0$ is constant, $\Delta$ and $K$ are the Laplacian and Gauss curvature
associated to the metric $g$ respectively. Moreover, we will assume that $V:= c + P
$, where $c $ is a nonnegative constant and $P $ satisfies
\begin{equation*}
|P(q)|\, \leq M/d(p_0,q) , 
\end{equation*}for some point $p_0\in \Sigma$, here $M$ is a nonnegative constant
\end{definition}

We continue assuming that $L_{a,c}$ has linear decay. In this case, we
obtain a stronger result.

\begin{theorem}\label{Theo:QDC}
Let $\Sigma $ be a complete Riemannian surface with $\partial \Sigma = \emptyset$. Suppose that $L_{a,c} = \Delta + V - a K$ is nonpositive acting on $f\in C^\infty _0(\Sigma)$, has linear decay, with $c > 0 $ and $a>1/4$. Then, $\Sigma$ is compact.\end{theorem}
\begin{proof}

Since $P$ has linear decay and $\Sigma$ is complete, there exists $s_0 >0$ so that $|P(q)| \leq c/2$ for all $q \in \Sigma \setminus D(p_0, s_0)$ (take $s_0 := 2M/c$). Set $\tilde \Sigma := \Sigma \setminus D(p_0,s_0)$, where $\partial \tilde \Sigma = \partial D(p_0,s_0)$.

On the one hand, from \cite[Theorem 2.8]{MPR}, the distance from every point $q \in \tilde \Sigma$ to the boundary of $\tilde \Sigma$ satisfies:
\begin{equation*}
d(q,\partial \tilde \Sigma) \leq \pi \sqrt{\left(1+\frac{1}{4a-1}\right)\frac{a}{c}} .
\end{equation*} 

On the other hand, the distance from every point $p \in D(p_0,s_0)$ to $\partial \tilde \Sigma$ is bounded by $2M/c$. That is, the distance from any two points in $\Sigma$ is uniformly bounded, depending only on $M$, $a$ and $c$. Therefore, since $\Sigma$ is complete, the Hopf-Rinow Theorem implies that $\Sigma$ is compact.
\end{proof}

\section{Stable surfaces in three-manifolds}\label{Sec:stable}

Let $\Sigma$ be a two-sided surface with constant mean curvature $H$ (in short,
$H-$surface) in a Riemannian three-manifold $\amb$. Throughout the rest of the paper, for the sake of simplicity, we will assume the ambient manifold $\amb$ is orientable without mention it. We will assume $\amb$ has bounded geometry, that is, $\amb$ has bounded sectional curvatures and injectivity radius bounded from below. $\Sigma$ is stable if (see \cite{SY} for the minimal case or \cite{BdCE} for the constant mean curvature case)
$$ \int _{\Sigma} \psi ^2 |A|^2 + \int _{\Sigma} \psi ^2 {\rm Ric}_{\amb} (N,N)
\leq \int_{\Sigma} |\nabla \psi|^2  $$for all compactly supported functions $\psi
\in H^{1,2}_c (\Sigma)$. Here $|A|^2$ denotes the the square of the length of the
second fundamental form of $\Sigma$, ${\rm Ric}_{\amb} (N,N)$ is the Ricci curvature
of $\amb$ in the direction of the normal $N$ to $\Sigma$ and $\nabla $ is the
gradient w.r.t. the induced metric.

One writes the stability inequality in the form
$$ \left.\frac{d^2}{dt^2}\right\vert _{t=0}\left({\rm Area}(\Sigma (t))-
2H {\rm Volume}(\Sigma (t))\right)= - \int _{\Sigma} \psi L \psi \geq 0 ,$$where $L$
is the linearized operator of the mean curvature
$$ L = \Delta + |A|^2 + {\rm Ric}_{\amb} .$$

In terms of $L$, stability means that $-L$ is nonnegative, i.e., all its eigenvalues
are nonnegative. $\Sigma$ is said to have finite index if $-L$ has only finitely
many negative eigenvalues. It is well known that the stability operator $L$ can be written as
\begin{equation*}
L = \Delta - K + (4H^2 - K_e +S),
\end{equation*}where $K$ and $K_e$ are the Gaussian curvature and extrinsic
curvature (i.e., the product of the principal curvatures) of $\Sigma$, and $S$ is
the scalar curvature of $\amb$. Hence, as a direct application of the previous
results, we have

\begin{theorem}\label{Cor:stable}
Assume $\amb$ has bounded geometry. Let $\Sigma \subset \amb$ be a complete oriented $H-$surface with finite index. Set $\eps \geq 0$ and
\begin{eqnarray*}
P^-_\eps &:=& {\rm min}\set{0 , 4H^2 - K_e +S -\eps} , \\
P^+_\eps &:=& {\rm max}\set{0 , 4H^2 - K_e +S -\eps}
\end{eqnarray*}

\begin{itemize}
\item If $P^-_0 \in L^1 (\Sigma)$, $\Sigma$ is conformally equivalent
to a compact Riemann surface with a finite number of points removed. Moreover, $P^+
_{0}\in L^1(\Sigma)$.
\item If $P^-_\eps \in L^1 (\Sigma)$ for some $\eps >0$, then $\Sigma$ is compact.
\end{itemize}
\end{theorem}
\begin{proof}
Since $-L$ has finite index and
\begin{equation*}
\begin{split}
L &= \Delta - K + (4H^2 - K_e + S) \\
 &= \Delta - K + \eps + P^+_\eps + P^- _\eps \\
 & \geq \Delta - K + \eps + P^-_{\eps},
\end{split}
\end{equation*}then $-L_\eps$, where $L_\eps := \Delta - K + \eps + P^-_{\eps}$, has
finite index. Thus, applying either Corollary \ref{Cor:QAG} if $\eps =0$ or
Corollary \ref{Cor:distance} if $\eps >0$, we obtain the result.

When $\eps >0$, we still have to remove the case when $\Sigma$ is parabolic with
finite volume. Since $\amb $ has bounded geometry and $\Sigma$ is complete and has
constant mean curvature, from \cite[Proposition 2.1]{CCZ}, we get that each end of
$\Sigma$ has infinite area, a contradiction. So, $\Sigma$ must be compact.

Let us see that $P^+_0 \in L^1 (\Sigma)$. Apply equation \eqref{MPRequation} for
$b=1$, then
\begin{equation*}
\int _{D(s)} V (1-r/s)^2 \leq 2a\pi - \frac{a(s)}{s^2},
\end{equation*}where $V:= 4H^2-K_e +S$. Set $V:= P_0 ^+ + P_0 ^-$.
Now, the above inequality can be written as
\begin{equation*}
1/2\int_{D(s/2)} P^+_0 \leq \int_{D(s)}P ^+ _0 (1-r/s)^2 \leq 2a\pi -
\frac{a(s)}{s^2} - \int_{D(s)}(1-r/s)^2 P^-_0,
\end{equation*}and since the right hand side of the above inequality is bounded as $s$ goes to infinity, we
get that $P^+_0 \in L^1(\Sigma)$.
\end{proof}

Also, we can drop the assumption about the bounded geometry of the ambient space in the next result:

\begin{theorem}\label{Cor:stableDecay}
Let $\Sigma \subset \amb$ be a complete oriented stable $H-$surface. Set $\eps > 0$ and
\begin{equation*}
P_\eps := 4H^2 - K_e +S -\eps .
\end{equation*}

Assume there exist $p_0 \in \Sigma$ and a constant $M>0$ so that
\begin{equation}\label{eqQCD}
|P_\eps (q)| \leq M /d(p_0,q).
\end{equation}

Then, $\Sigma$ is compact
\end{theorem}
\begin{proof}
Note that the stability operator $L$ can be written as
$$ L = \Delta - K + \eps + P_\eps .$$

Then, under the assumption \eqref{eqQCD}, $L$ has linear decay (see Definition
\ref{Def:QDecay}). Hence, applying Theorem \ref{Theo:QDC}, we obtain the result.
\end{proof}

We focus now on surfaces immersed on a Killing submersion $\amb$, that is, $\amb$ is
a Riemmanian submersion over a Riemannian surface $\m$ whose fibers are the
trajectories of an unit Killing field. In \cite{ES}, the geometry of this kind of
submersion is studied. In some sense, these spaces behaves like a simply-connected
homogeneous space $\hmf$.

Let $\amb$ be a three-dimensional Killing submersion, then $\pi : \amb \to \m$ over a surface $(\m , g)$ with Gauss curvature $\kappa$, and the {\it fibers}, i.e. the inverse image of a point at $\m $ by $\pi$, are the trajectories of a unit Killing vector field $\xi $, and hence geodesics.
Denote by $\meta{}{}$, $\camb$, $\ext $, $\bar R$ and $[,]$ the metric, Levi-Civita
connection, exterior product, Riemann curvature tensor and Lie bracket in $\amb$,
respectively. Moreover, associated to $\xi$, we consider the operator $J: \campo
(\amb) \to \campo (\amb)$ given by
\begin{equation*}
J X : = X \ext \xi , \, \, \, X \in \campo (\amb).
\end{equation*}

Given $X \in \campo (\amb)$, $X$ is {\it vertical} if it is always tangent to
fibers, and {\it horizontal} if always orthogonal to fibers. Moreover, if $X \in
\campo (\amb) $, we denote by $X^v$ and $X^h$ the projections onto the subspaces of
vertical and horizontal vectors respectively. In particular (see \cite[Proposition 2.6]{ES})

\begin{proposition}\label{Prop:tau}
Let $\amb$ be as above. There exists a function $\tau : \amb \to \r $ so that
\begin{equation}
\camb _X \xi = \tau \, X \ext \xi , \,
\end{equation}here $\camb$ denotes the Levi-Civita connection on $\amb$.
\end{proposition}

Actually, it can be shown that $\tau$ only depends on $\m$ (this is a personal communication of A. Jimenez, and it will appear in a forthcoming paper). This makes natural the following notation:

\begin{definition}
A Riemannian submersion over a surface $\m $ whose fibers are the trajectories of an
unit Killing vector field $\xi$ will be called {\it Killing submersion} and denoted
by $\hm $, where $\kappa $ is the Gauss curvature of $\m $ and $\tau $ is given in
Proposition \ref{Prop:tau}.
\end{definition}

Moreover, we remind here \cite[Lemma 2.8]{ES}

\begin{lemma}\label{Lem:SectK}
Let $\hm $ be a Riemannian submersion with unit Killing vector field $\xi $. Let
$\set{X, Y} \in T \hm $ be an orthonormal basis of horizontal vector fields so that
$\set{X,Y,\xi}$ is positively oriented. Then
\begin{eqnarray}
\bar K( X\ext Y ) &=& \kappa - 3\tau ^2 , \label{SectHH}\\
\bar K( X \ext \xi ) &=& \tau ^2 . \label{SectHV}
\end{eqnarray}

Moreover, the scalar curvature $S$ of $\hm$ at $p\in \hm$ is given by
\begin{equation}\label{scalarcurv}
S(p) = \kappa - \tau ^2 .
\end{equation}
\end{lemma}

Now, we can announce:

\begin{theorem}\label{Theo:stablehm}
Let $\Sigma $ be a complete oriented $H-$surface with finite index immersed in $\hm$, $\hm$ a Killing submersion of bounded geometry so that $4H^2 +c(\Sigma)\geq 0$, where
\begin{equation*}
c(\Sigma) := {\rm inf}\set{\kappa (\pi (p)) : \, p \in \Sigma} .
\end{equation*}

Set
\begin{eqnarray*}
P^- &:=& {\rm min}\set{0 , -(K_e +\tau^2) } , \\
P^+ &:=& {\rm max}\set{0 , -(K_e +\tau^2)  }.
\end{eqnarray*}

Assume $P^- \in L^1 (\Sigma). $Then, one of the following statements hold:
\begin{itemize}
\item $\Sigma$ is a minimal graph with $\pi(\Sigma)=\m$ and $c(\Sigma)>0$,
\item $4H^2 + c(\Sigma)=0$ and $\Sigma$ is either a vertical multigraph or
vertical cylinder of geodesic curvature $2H$ in $\m$.
\end{itemize}
\end{theorem}
\begin{proof}
The linearized operator for the mean curvature is given by
$$ L := \Delta - K +(4H^2 - K_e + S) .$$

Now, from \eqref{scalarcurv}, we might rewrite the above inequality as
\begin{equation*}
L:= \Delta - K + 4H^2 + \kappa - (K_e + \tau ^2)\geq \Delta - K +c + P^- ,
\end{equation*}where $c:= 4H^2 + c(\Sigma) \geq 0 $. Then $-\tilde L$, where
$\tilde L := \Delta - K + c + P^-$, is nonnegative acting on $f\in C^\infty _0
(\Sigma)$ and has integrable potential.

If $4H^2+c(\Sigma)>0$, applying Theorem \ref{Theo:distancegeq}, $\Sigma$ is either
compact or parabolic with finite area.

Let us prove now that $\Sigma$ can not be parabolic with finite area. Since $\hm$
has bounded geometry and $\Sigma$ is complete and has constant mean curvature, from
\cite[Proposition 2.1]{CCZ}, we get that $\Sigma$ has infinite area. So, $\Sigma$
must be compact.

Set $\nu := \meta{\xi }{N}$, where $N$ is the unit normal vector field along
$\Sigma$, $\nu$ is a bounded Jacobi function, i.e., $L\nu =0$. Since $\Sigma$ is
stable and compact, elementary elliptic theory asserts that either $\nu$ vanishes
identically or $\nu >0 $.

If $\nu $ vanishes identically, $\Sigma := \pi ^{-1}(\alpha)$, i.e., it is a
vertical cylinder over a complete curve $\alpha \subset \m$ of geodesic curvature
$2H$. Then, the stability operator of $\Sigma := \pi ^{-1}(\alpha)$ is given by
\begin{equation*}
L = \Delta + (4H^2 + \kappa)
\end{equation*}since $K_e = -\tau^2$ on a vertical cylinder (see \cite[Proposition 2.10]{ES}). Then
\begin{equation*}
L \geq \Delta + c ,
\end{equation*}where $c:= 4H^2 + c(\Sigma) >0$.

Let $\lambda _1 (L)$ denote the first eigenvalue of $-L$. Thus,
\begin{equation*}
\lambda _1 (L ) \leq \lambda _1 (\Delta + c) = \lambda _1 (\Delta) -c = -c <0 ,
\end{equation*}where $\lambda _1 (\Delta) =0$ since $\Sigma$ is isometrically a
plane. So, $\Sigma$ is unstable, a contradiction.

Therefore, $\nu$ never vanishes. Since $\Sigma$ is compact, there exists $\eps >0$ so that
$\nu \geq \eps >0$, and hence $\pi (\Sigma) = \m $. Let $\jmath (\m )$ denote the Cheeger constant, i.e.,
$$ \jmath (\m) := {\inf}_{\Omega \subset \m}\set{\frac{A(\Omega)}{L(\partial \Omega)}},
$$where $\Omega $ varies over open domains on $\m$ with compact closure and smooth
boundary.

Let $\Omega \subset \m$ be a relatively compact domain with smooth boundary
$\partial \Omega$. Since $\nu \geq \eps >0$, there exists a compact set $\Sigma _0
\subset \Sigma $ which is a $H-$graph over $\Omega$. From the Divergence Theorem
$$ 2H A(\Omega) = \int_ \Omega  {\rm div}(N^h) = \int _{\partial \Omega}
g(N^h , \eta) \leq L(\partial \Omega),$$where $A(\Omega)$ and $L(\partial \Omega)$
are the area and the length of $\Omega$ and $\partial \Omega$ (w.r.t. $g$)
respectively. Moreover, ${\rm div}$ is the divergence operator on $(\m ,g)$. Thus,
\begin{equation}\label{Cheeger}
2H \leq \jmath (\m ).
\end{equation}

Since $\m$ is compact (recall $\pi(\Sigma) = \m$ and $\Sigma $ is compact),
$\jmath(\m)=0$. So, from \eqref{Cheeger} and $4H^2 + c(\Sigma)>0$, we get $H =0$ and
$c(\Sigma)>0$.

If $4H^2+c(\Sigma)=0$, $\Sigma$ is parabolic. Then, \cite[Corollary 2.5]{MaPRo}
asserts that $\nu$ vanishes identically or never vanishes. That is, $\Sigma$ is
either a vertical cylinder over a complete curve of geodesic curvature $2H$ in $\m$,
or $\Sigma$ is a complete multigraph.

This finishes the proof.
\end{proof}

Also,

\begin{theorem}\label{Theo:stablehmQCD}
Let $\Sigma $ be a complete oriented stable $H-$surface in $\hm$ so that $4H^2 +
c(\Sigma) > 0$, where
\begin{equation*}
c(\Sigma) := {\rm inf}\set{\kappa (\pi (p)) : \, p \in \Sigma} .
\end{equation*}

Set $P := K_e +\tau ^2$, assume there exist a point $p_0 \in \Sigma$ and a constant $M>0$ so that
\begin{equation*}
|P(q)| \leq M/d(p_0,q).
\end{equation*}

Then, $\Sigma$ is a minimal graph with $\pi(\Sigma)=\m$ and $c(\Sigma) >0$.
\end{theorem}
\begin{proof}
We might write the stability operator $L$ as
\begin{equation*}
L = \Delta - K + (4H^2 +\kappa) - (K_e + \tau^2),
\end{equation*}then $L \geq \tilde L$, where
\begin{equation*}
\tilde L := \Delta - K + c - P ,
\end{equation*}where $c:=4H^2 +c(\Sigma)>0$. Thus, $\tilde L$ is a nonnegative operator
with linear decay. So, from Theorem \ref{Theo:QDC}, $\Sigma$ is compact. Therefore, arguing as in Theorem \ref{Theo:stablehm}, we obtain that $H=0$, $c(\Sigma)>0$ and $\pi (\Sigma) = \m$.
\end{proof}

To finish, we give some consequences of the above results for stable $H-$surface in
$\hmf$. Nelli-Rosenberg \cite{NR} proved that there are no stable $H-$surfaces,
$H>1/\sqrt{3}$, in $\hr$, by proving a {\it Distance Lemma}. In general, thanks to a
{\it Distance Lemma}, i.e. an intrinsic estimate of the distance to the boundary,
Rosenberg \cite{R2} proved that there are no complete stable $H-$surfaces in $\hmf$
provided $H^2 > \frac{\tau ^2 - \kappa}{3}$, unless $\s^2 (\kappa)\times \set{0}$ in
$\s ^2 (\kappa)\times \r$. Moreover, we might assume that $\kappa < \tau ^2$, since
the case $\kappa \geq \tau ^2 $ is done (see \cite{R2}), a stable complete
$H-$surface in $\hmf$, $\kappa \geq \tau ^2$ is a slice in $\s ^2 (\kappa )\times
\r$. When, $\kappa < \tau ^2 $, one can improves Rosenberg's result using a
compactness argument. In \cite{MaPRo}, the authors proved that there exists $\eps
>0$ such that $H^2< \frac{\tau ^2 - \kappa}{3}-\eps$ for any complete stable
$H-$surface in $\hmf$ with $\kappa < \tau^2$.

Here, we prove that there are no complete $H-$surfaces, $H>1/2$, in $\hr$ under some
conditions on the extrinsic curvature.

\begin{corollary}\label{Cor:hr}
Let $\Sigma \subset \hr$ be an oriented complete stable $H-$surface satisfying one of the following conditions:
\begin{itemize}
\item $H\geq 1/2$ and ${\rm max}\set{0,K_e} \in L^1 (\Sigma)$,
\item $H>1/2$ and there exist a point $p_0 \in \Sigma$ and a constant $M>0$ so that
\begin{equation*}
|K_e (q)| \leq M/d(p_0,q).
\end{equation*}
\end{itemize}

Then, $H=1/2$ and $\Sigma$ is either a vertical horocylinder (i.e. a vertical
cylinder over a horocycle in $\h^2$) or an entire vertical graph.
\end{corollary}
\begin{proof}
We apply either Theorem \ref{Theo:stablehm} or Theorem \ref{Theo:stablehmQCD} depending on
the condition that $\Sigma$ verifies. In any case, $4H^2+1=0$ and $\Sigma$ is either
a vertical cylinder over a complete curve of geodesic curvature $2H=1$, that is, a
horocycle in $\h ^2$, or $\Sigma$ is a complete multigraph. In the latter case,
Hauswirth-Rosenberg-Spruck \cite{HRS} proved that $\Sigma$ is an entire graph.
\end{proof}

And, for stable $H-$surfaces in either the Heisenberg space or $\widetilde{{\rm
PSL}(2,\r)}$, we obtain

\begin{corollary}\label{Cor:homog}
Let $\Sigma \subset \hmf $, $\tau \neq 0$, be an oriented complete stable
$H-$surface satisfying one of the following conditions:
\begin{itemize}
\item $4H^2 +\kappa \geq 0$ and $\nu^2 \in L^1 (\Sigma)$,
\item $4H^2 +\kappa > 0$ and there exist a point $p_0 \in \Sigma$ and a constant $M>0$ so that
\begin{equation*}
|\nu (p)|^2 \leq M/d(p_0,q).
\end{equation*}
\end{itemize}

Then:
\begin{itemize}
\item In $\s ^3_{Berger}$, there are no such a stable $H-$surface.
\item In ${\rm Nil}_3$, $H=0$ and $\Sigma$ is either a vertical plane (i.e. a vertical cylinder over a
straight line in $\r^2$) or an entire vertical graph.
\item In $\widetilde{{\rm PSL}(2,\r)}$, $H=\sqrt{-\kappa}/2$ and $\Sigma$ is either a vertical
horocylinder (i.e. a vertical cylinder over a horocycle in $\h ^2 (\kappa)$) or an entire graph.
\end{itemize}
\end{corollary}
\begin{proof}
Note that, by the Gauss equation for a surface immersed in $\hmf$ (see \cite{D}),
i.e.,
$$ K = K_e + \tau ^2 + (\kappa - 4\tau ^2)\nu ^2 ,$$we might write the stability
operator as
$$ L = \Delta - 2K + (4H^2+\kappa) + (\kappa - 4\tau ^2)\nu ^2 .$$

So, we apply Theorem \ref{Theo:stablehm} or Theorem \ref{Theo:stablehmQCD} depending
on the condition that $\Sigma$ verifies. We get:

\begin{itemize}
\item In $\s ^3 _{Berger}$, $4H^2 + \kappa >0$ and so $\Sigma$ is compact, but there are no
compact, oriented stable $H-$surfaces in $\s ^3 _{Berger}$ (see \cite[Corollary
9.6]{MPR}).

\item In ${\rm Nil}_3$, $H=0$ and $\Sigma$ is either a vertical cylinder over a complete
curve of geodesic curvature $H=0$, that is, a straight line in $\r^2$, or $\Sigma$
is a complete multigraph. In the latter case, Daniel-Hauswirth \cite{DH} proved that
$\Sigma$ is an entire graph.

\item In $\widetilde{{\rm PSL}(2,\r)}$, the proof is similar as above. Now, $\Sigma$ is an entire graph follows from \cite{DHM}.
\end{itemize}
\end{proof}

\section{Appendix}

We recall here some results we have used along the paper for the sake of
completeness. The first one is a general inequality for $I(f)$ (see
\eqref{varcharacsurface}) following the method developed by T. Colding and W.
Minicozzi in \cite{CM}. We establish here the formula how it was stated in
\cite[Lemma 3.1]{ER}, but the proof can be found in \cite{Ca}.

We denote
\begin{eqnarray*}
l(s) &=& {\rm Length}(\partial D(s)) \\
a(s) &=& {\rm Area}(D(s))\\
K(s) &=& \int _{D(s)} K \\
\chi (s)&=& \text{Euler characteristic of } D(s)
\end{eqnarray*}

\begin{lemma}[Colding-Minicozzi stability inequality]\label{Lem:ColdingMinicozzi}
Let $\Sigma $ be a Riemannian surface possibly with boundary and $K \not \equiv 0$.
Let us fix a point $p_0\in \Sigma$ and positive numbers $0 \leq \eps < s $ such that
$\overline{D(s)}\cap \partial \Sigma = \emptyset$. Let us consider the differential
operator $L_a = \Delta + V - a K$, where $V \in C^{\infty}(\Sigma)$ and $a$ is a
positive constant, acting on $f \in C^{\infty}_0 (\Sigma)$. Let $f: D(s) \To \r $ a
nonnegative radial function, i.e. $f \equiv f(r)$, such that

$$\begin{matrix}
f(r) \equiv  1, & \text{ for } r \leq \eps \\
f(r) \equiv  0, & \text{ for } r \geq s \\
f'(r) \leq 0 , & \text{ for } \eps < r < s
\end{matrix}$$

Then, the following holds
\begin{equation}\label{CMineq}
 I(f) \leq 2a \left( \pi G(s) - f'(\eps) l(\eps)\right) -\int _{D(s)} V f(r)^2
 + \int _{\eps}^s \left\{ (1-2a)f'(r) ^2 - 2 a f(r)f''(r)\right\}l(r) ,
\end{equation}where
\begin{equation*}
G(s) := -\int _\eps ^s (f(r)^2)' \chi (r) .
\end{equation*}
\end{lemma}

Moreover, we should mention that we have relaxed here the hypothesis on $V$, but
from the proof, we can see that we do not really use the fact that $V\geq 0$.

It is well known that the kind of results we can obtain for nonnegative operator of
the form
$$ L_a : = \Delta + V - a \, K $$where $V \geq 0$ and $a$ is a positive constant,
depend strongly on the value of $a$.

The most studied case is when $a > 1/4$ (see \cite{Ca} or \cite{MPR}). When $a>1/4$,
we use the following radial function
\begin{equation}\label{functionageq14}
f(r)=\left\{\begin{matrix}
\left( 1-\frac{r}{s}\right) ^ b & 0 \leq r \leq s \\[3mm]
0 & r \geq s
\end{matrix}\right.
\end{equation}where $s>0$, $b\geq 1$ and $r$ is the radial distance of a point $p$ in $D(s)$ to
$p_0$. So now, we establish a formula developed by Meeks-P\'{e}rez-Ros \cite{MPR}.
Such a formula follows from Lemma \ref{Lem:ColdingMinicozzi} with the test function
given by \eqref{functionageq14}.

\begin{corollary}\label{Cor:ageq14}
Let $\Sigma $ be a Riemannian surface possibly with boundary and $K \not \equiv 0$.
Fix a point $p_0\in \Sigma$ and a positive number $ s > 0$ such that
$\overline{D(s)}\cap \partial \Sigma = \emptyset$. Suppose that the differential
operator $L_a = \Delta + V - a K$ is nonpositive on $ C^{\infty}_0 (\Sigma)$, where
$V \in C^{\infty}(\Sigma)$ and $a > 1/4$ is a constant. For $b\geq 1$, we have
\begin{equation}\label{MPRequation}
\int _{D(s)} \left( 1- r/s\right)^{2b} V \leq 2a G (s) \pi
+\frac{b(b(1-4a)+2a)}{s^2}\int _{0}^s \left( 1- r/s\right)^{2b-2} l(r) ,
\end{equation}where
\begin{equation}\label{G}
G(s):= \frac{2b}{s}\int _0 ^s \left( 1-r/s \right)^{2b-1} \chi (r) \, dr \leq 1.
\end{equation}
\end{corollary}

If $a\leq 1/4$ (see \cite{ER}), we will work with the special radial function given
by

\begin{equation}\label{function}
f(r)=\left\{\begin{matrix}
1 & r \leq s e^{-s}  \\[3mm]
\left( \dfrac{\ln (s/r)}{s}\right) ^ b & s e^{-s} \leq r \leq s \\[3mm]
0 & r \geq s
\end{matrix}\right.
\end{equation}where $s>0$, $b\geq 1$ and $r$ is the radial distance of a point $p$ in $D(s)$ to
$p_0$. Now, we use the above test function \eqref{function} in Lemma
\ref{Lem:ColdingMinicozzi} (see \cite[Corollary 6.1 and Theorem 6.1]{ER} for
details).

\begin{corollary}\label{Cor:aleq14}
Let $\Sigma $ be a Riemannian surface with $k-$AAB, possibly with boundary and $K
\not \equiv 0$. Fix a point $p_0\in \Sigma$ and a positive number $ s > 0$ such that
$\overline{D(s)}\cap \partial \Sigma = \emptyset$. Suppose that the differential
operator $L_a = \Delta + V - a K$ is nonpositive on $ C^{\infty}_0 (\Sigma)$, where
$V \in C^{\infty}(\Sigma)$ and $0< a \leq 1/4$ is a constant. Let $b\geq 1$ so that
$2(b+1)\geq k$, then
\begin{equation}\label{estVpos}
\int _{D(s e^{-s})}  V  + \int _{D(s)\setminus D(s e^{-s})} \left(
\frac{\ln(s/r)}{s}\right) ^{2b} V  \leq
 2a \left(\pi + b\frac{2\pi - \mathcal{K}(se^{-s})}{s}\right) + \rho^{+}_{a,b}(\delta _0
 ,s),
\end{equation}where we denote
$$ \mathcal{K}(r_1) = {\rm min}_{[0,r_1]}\set{K(r)} , $$and $\rho ^{+}_{a,b}$ is a
function depending on $s$ so that
\begin{equation*}
\rho^{+}_{a ,b}(\delta_0 , s) \To \left\{ \begin{matrix}
0 & \text{for} & 2(b+1)> k \\
C^+ & \text{for} & 2(b+1)= k
\end{matrix}\right.  \, \text{ if } \, s \to + \infty,
\end{equation*}for a positive constant $C^+$. Moreover,
\begin{equation*}
\lim _{s\to \infty} \frac{\mathcal{K}(s e^{-s})}{s} = 0.
\end{equation*}
\end{corollary}

\begin{center}
{\bf Acknowledgement:}\\
{\it The author wishes to thank J. P\'{e}rez, A. Ros and H. Rosenberg for their interesting comments and help during the preparation of this work. Also, the author wishes to thanks to the referees for their corrections and insights for improving this paper.}
\end{center}

\end{document}